\documentclass[aps,groupedaddress,11pt,tightenlines,nofootinbib]{revtex4}
\usepackage{amsmath,amssymb,amsthm}
\usepackage[dvips]{graphicx}

\newtheorem{Theorem}{Theorem}
\newtheorem{Cor}{Corollary}
\newtheorem{Prop}{Proposition}

\def\ra{\rightarrow}

\def\R{\mathbb{R}}
\def\Rp{\mathbb{R}_+}

\def\d{\partial}
\def\w#1{\widetilde{#1}}
\newcommand{\fw}{\ensuremath{\;|\;}}
\newcommand{\E}{\ensuremath{\mathcal{E}}}
\DeclareMathOperator{\Div}{div}
\DeclareMathOperator{\Flux}{Flux}

\begin{document}

\title{Global pointwise decay estimates for 
\\defocusing radial nonlinear wave equations
}

\author{Roger Bieli}
\author{Nikodem Szpak}
\affiliation{Max-Planck-Institut f\"{u}r
Gravitationsphysik, Albert-Einstein-Institut, Golm, Germany}
\date{\today}

\begin{abstract}
We prove global pointwise decay estimates for a class of defocusing semilinear wave equations in $n=3$ dimensions restricted to spherical symmetry. The technique is based on a conformal transformation and a suitable choice of the mapping adjusted to the nonlinearity. As a result we obtain a pointwise bound on the solutions for arbitrarily large Cauchy data, provided the solutions exist globally. The decay rates are identical with those for small data and hence seem to be optimal. A generalization beyond the spherical symmetry is suggested.
\end{abstract}


\maketitle

\section{Introduction}

We consider a class of nonlinear wave equations
\begin{equation} \label{eq:gwav}
  \d_t^2 \phi - \Delta \phi = f(\phi)
\end{equation}
in $n=3$ spatial dimensions where the nonlinearity $f(\phi)$ is of a defocusing type. It means that the nonlinear term has a repulsive action on the waves and focusing of waves is suppressed by an energy condition. Essential is the sign of the nonlinear term which is chosen such that the following conserved energy is positive definite
\begin{equation} \label{eq:gerg}
  E=\int_{\R^3} \left(\frac{1}{2}|\d_t \phi|^2 + \frac{1}{2}|\nabla \phi|^2+F(\phi)\right)\,d^3 x
\end{equation}
with $f=-F'$.
Equation \eqref{eq:gwav} has been intensively studied in the literature over a few decades, in particular in the case of a pure power nonlinearity $f(\phi)=-|\phi|^{p-1}\phi$.
Let us collect the most important results defining the context for this work.
The global existence of $C^2$ solutions to the Cauchy problem has been shown by J{\"o}rgens \cite{Joergens} in the energy subcritical case $1<p<5$ and later by Grillakis \cite{Grillakis} in the critical case $p=5$ (see also \cite{Struwe} for the spherically symmetric $p=5$ case) while not much is known about the supercritical case $p>5$. For more references related to global existence we recommend the book of Sogge \cite{Sogge-book}.
Uniform boundedness of solutions for $2<p<5$ has been proved by Pecher \cite{Pecher}.
Uniform decay $1/t^{1-\epsilon}$ and scattering have been proved for $3\leq p<5$ by Strauss \cite{Strauss_SemilinDecay}.
Bahouri and Shatah \cite{BahShat} have shown that finite energy solutions decay to zero for $p=5$ and
Hidano \cite{Hidano_SemilinScatt} has shown scattering and decay to zero for $2.5<p\leq 3$.
Ginibre and Velo \cite{GinVelo_scatt} have shown scattering in the energy space for $1<p<5$ in various dimensions, but in $n=3$ only for small data.
Scattering results imply, in some sense, that the solutions behave asymptotically like solutions of the linear equation. (For more scattering results we refer to the monograph of Strauss \cite{Strauss-book}.)


Here, we go further and study the pointwise behavior of solutions for $3\leq p<5$. Such results in dimension $n=3$ exist only for small data \cite{Asakura, Strauss-T, NS-Tails, NS-PB_Tails} and are based on perturbation techniques which cannot be generalized to large data. We extend the technique of conformal compactification developed by Choquet-Bruhat, Christodoulou and others in \cite{ChoqBruh-P-Segal,Christodoulou-Quasilin,Baez-Segal-Zhou}.
The novelty is that we use a suitably chosen conformal transformation, adjusted to the nonlinearity. In this setting, we first show uniform boundedness of the transformed solutions in the precompact region and then use the inverse transformation to get an extra decaying (conformal) factor and thus a pointwise decay estimate for the original solution. We claim that this estimate is optimal (for generic initial data) as it is identical with the one obtained for small data where it has been shown to be optimal \cite{NS-Tails, NS-PB_Tails}. Numerical simulations done by Bizo{\'n} \textit{et al} \cite{PB-private} also support this picture.

\subsection*{The main result and the method}

In the following we study the Cauchy problem for the radial semilinear wave equation
\begin{equation} \label{eq:wave-p}
  \d_t^2 \phi - \Delta \phi = -|\phi|^{p-1}\phi
\end{equation}
with the initial data of compact support $r\in[0,\alpha_p[$, where $\alpha_p>0$ will be specified later\footnote{Symmetries of the equation \eqref{eq:wave-p} allow for mapping any compactly supported initial data onto the interval $r\in[0,\alpha_p[$ at $t=1$.}, and chosen at time $t=1$
\begin{align}
  \phi(1,r)&=\phi_0,& \d_t\phi(1,r)&=\phi_1
\end{align}
and prove that the solutions satisfy for all $t\geq 1$ and $0\leq r \leq t$
\begin{equation} \label{eq:main-estimate}
  |\phi(t,r)| \leq \frac{C}{(1+t+r)(1+t-r)^{p-2}}
\end{equation}
with some constant $C$ depending only on $\phi_0, \phi_1$ and $p$.

Briefly, our method is based on a conformal mapping $(u,v)\ra(\w{u},\w{v})$ defined in the double-null coordinates $u=t+r$, $v=t-r$,
\begin{equation}\label{eq:uv-map}
  \w{u} := -\frac1{u^{p-2}},\qquad \w{v} := -\frac1{v^{p-2}}.
\end{equation}
The conformal factor $\Omega := \frac{\w{r}}{r}$ multiplies the transformed solution
$ \w{\phi}(\w{u},\w{v}) := \Omega^{-1} \phi(u,v) $
which satisfies the transformed wave equation
\begin{equation} \label{eq:wave-mapped}
 \w{\Box} \w{\phi} + \frac1{(p-2)^2} \left[\frac{\w{u}-\w{v}}{(-\w{v})^{1/(p-2)} - (-\w{u})^{1/(p-2)}} \right]^{p-1} \w{\phi} |\w{\phi}|^{p-1} = 0
\end{equation}
in a precompact region of spacetime. There, we are able to show boundedness of some pseudo-energy flux what we further use to show the uniform boundedness of $|\w{\phi}(\w{u},\w{v})|\leq \w{C}$. Finally, the inverse transformation provides the desired estimate
$$ |\phi(u,v)| \leq \w{C}\cdot \Omega = \w{C} \cdot \frac{1}{(uv)^{p-2}} \cdot \frac{u^{p-2}-v^{p-2}}{u-v} \leq \frac{C}{u\cdot v^{p-2}}. $$
The power $p-2$ in the mapping \eqref{eq:uv-map} cannot be increased, what would potentially lead to a stronger pointwise decay of the solutions, because then the factor multiplying the nonlinearity in \eqref{eq:wave-mapped} becomes singular and our boundedness theorems cannot be applied.

The estimate \eqref{eq:main-estimate} is also optimal in the sense of compatibility with the small data case. For small initial data it has been shown in \cite{NS-PB_Tails} that the solutions behave asymptotically, for large $t$ and fixed $r$, like
\begin{equation*}
  \phi(t,r) = \frac{C}{t^{p-1}} + {\cal{O}}(t^{-p})
\end{equation*}
where the constant $C$ can be expressed explicitly via $\phi_0$ and $\phi_1$. Hence, the estimate \eqref{eq:main-estimate} applied to the small data case gives the optimal decay rate.

The assumption of spherical symmetry is essential for the method to yield
optimal decay for $p>3$. On $\R \times \R^3$ the obvious analogue of the map
$(u,v)\ra(\w{u},\w{v})$ is no longer conformal, unless $p=3$, of course. The
reason is that radial and angular parts of the wave operator transform
diversely and the resulting equation is no longer semilinear. Applying the
transformation corresponding to $p=3$, which is conformal, to an equation
with nonlinearity of power $p>3$ still gives a global $1/(uv)$-decay result,
being however not optimal and independent of the actual power $p$.

\section{Conformal transformation} \label{sec:ct}

In the case $p=3$, the radial wave equation \eqref{eq:wave-p}, having
the explicit form \[ \d_t^2 \phi - \d_r^2 \phi - \frac2r \d_r \phi +
\phi^3 = 0, \] is not only invariant under time translation, $\phi(t,r)
\mapsto \phi(t+a,r)$, scaling, $\phi(t,r) \mapsto \lambda \phi(\lambda
t,\lambda r)$, and reflection, $\phi \mapsto -\phi$, but also under a
\emph{conformal inversion} which is given by
\begin{equation} \label{eq:convinv}
 \phi(t,r) \mapsto \frac1{t^2-r^2} \phi \biggl( \frac{-t}{t^2-r^2},
 \frac r{t^2-r^2} \biggr).
\end{equation}
This inversion maps solutions on the forward light-cone $K^+:=\{(t,r)
\fw 0\leq r < t\}$ to solutions on the backward light-cone $K^-:=\{(t,r)
\fw 0\leq r < -t\}$ of the origin and vice versa. Thus, establishing
boundedness of a solution on $K^-$ towards the future immediately implies
pointwise decay estimates for the transformed solution on $K^+$. In order
for this to yield optimal decay in the case of more general nonlinearities
it is necessary to also consider more general conformal transformations
which is what will be discussed in the following.

On $(t,r)\in\R\times\Rp$ consider null coordinates $u:=t+r$ and $v:=t-r$.
Then $u$ and $v$ are positive on $K^+$ and negative on $K^-$. For any $p>2$
define, with respect to the $(u,v)$-coordinate system, the map \[ \Phi:
K^+ \to K^-,\quad (u,v) \mapsto \biggl( -\frac1{u^{p-2}}, -\frac1{v^{p-2}}
\biggr). \] It is analytic with analytic inverse \[ \Phi^{-1}: K^- \to K^+,\
(u,v) \mapsto \biggl( (-u)^{-\frac1{p-2}}, (-v)^{-\frac1{p-2}} \biggr).\]
Furthermore, there exists a positive analytic function $\Omega>0$ on $K^+$
with the property that
\begin{equation} \label{eq:om}
 r \Omega = r \circ \Phi
\end{equation}
holds on $K^+$.  If $\eta = dt^2 - dr^2$ denotes the Minkowski metric
on $\R\times\Rp$ then its pullback by $\Phi$ satisfies \[ \Phi^* \eta
= \frac{(p-2)^2}{(t^2-r^2)^{p-1}} \eta \quad \textnormal{on } K^+,\]
which shows that mapping by $\Phi$ is conformal. In the case $p=3$ the
map $\phi \mapsto \Omega \Phi^* \phi=\Omega \cdot (\phi \circ \Phi)$
corresponds to the conformal inversion~\eqref{eq:convinv} since then
$\Omega = 1/(t^2-r^2)$. This conformal transformation is illustrated in
Figure~\ref{fig:ct}. Note that with \[ \alpha_p:=1-\frac1{2^\frac1{p-2}} \in
\ ]0,1[ \] the compact interval $I$ contained in $\{1\} \times [0,\alpha_p[$
implies that $J$ is compactly contained in $\{-1\} \times [0,1[$, where
$J$ is the intersection of the future of the curve $H := \Phi(I)$ with
the line $\{t=-1\}$. Moreover, for $p\geq 3$ one has $\alpha_p \leq 1/2$,

\begin{figure}
 \includegraphics[width=0.3\textwidth]{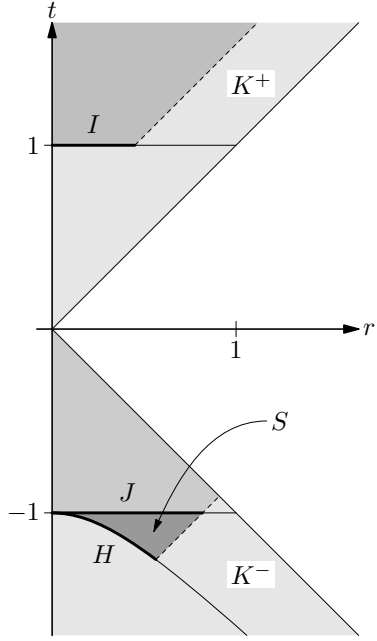}
 \caption{The function $\Phi$ maps the forward light-cone $K^+$ to the
  backward light-cone $K^-$ conformally and bianalytically. An interval $I$
  of the line $\{t=1\}$, containing the support of the original Cauchy data,
  is mapped to a curve $H$.  The future of the interval $I$ is mapped into
  the future of the curve $H$. Initial data on $H$ is evolved in the region
  $S$ to yield new initial data on the interval $J$ of the line $\{t=-1\}$.}
 \label{fig:ct}
\end{figure}

Suppose that $h \in C^2(K^-)$ is a twice continuously differentiable
function on $K^-$. Then $\Omega \Phi^* h$ is a function on $K^+$ of class
$C^2$ for which \[ \Box ( \Omega \Phi^* h ) = (\Box \Omega)\Phi^* h + \Omega
\frac{(p-2)^2}{(t^2-r^2)^{p-1}} \Phi^* \Box h, \] From relation~\eqref{eq:om}
it follows that $\Box \Omega = 0$ such that
\begin{equation} \label{eq:ct}
 \Box( \Omega \Phi^* h ) = (p-2)^2 (uv)^{-(p-1)} \Omega \Phi^*
 \Box h.
\end{equation}
Hence, if $p>2$ and $\phi \in C^2(K^+)$ is a classical solution of the
radial wave equation
\begin{equation} \label{eq:wav}
 \Box \phi + \phi |\phi|^{p-1} = 0
\end{equation}
on $K^+$, its conformal transformation $\psi:=\Phi_*(\Omega^{-1}
\phi) \in C^2(K^-)$ satisfies
\begin{equation} \label{eq:ctwav}
 \Box \psi + \frac1{(p-2)^2} \bigl[ \Phi_* (\Omega u v)^{p-1} \bigr] \psi
 |\psi|^{p-1} = 0,
\end{equation}
where $\Phi_*$ denotes the push-forward by the diffeomorphism
$\Phi$. On the other hand, it follows directly from \eqref{eq:ct}
that if $\psi \in C^2(K^-)$ is a classical solution of the transformed
equation~\eqref{eq:ctwav} on $K^-$, the function $\phi:= \Omega \Phi^*
\psi \in C^2(K^+)$ solves the original equation~\eqref{eq:wav} on $K^+$.

\section{Evolution of the pseudo-energy} \label{sec:evolE}

As mentioned introductorily, for equations of the form \eqref{eq:gwav} there
is a \emph{conserved} energy~\eqref{eq:gerg} that in particular controls the
spatial $L^{p+1}$-norm of the solution. One can no longer expect to find
such a quantity for the conformally transformed equation~\eqref{eq:ctwav}
or, more generally, for an equation
\begin{equation} \label{eq:vwav}
 \Box \psi + c(t,r) \psi |\psi|^{p-1} = 0
\end{equation}
with a non-negative function $c$ of class $C^2$. However, an analogue of
the energy expression, although not conserved exactly, will turn out to be
sufficient to prove boundedness of $\psi$ on the relevant region of $K^-$,
provided the function $c$ is monotonically decreasing in $t$ there.
This pseudo-energy will now be defined.

Let $\psi \in C^2(K^-)$ be a classical solution of \eqref{eq:vwav} on
$K^-$ for $p>2$.  Then the vector field $\E$ given by \[ \E:= r^2 \biggl[
\frac12 (\d_t \psi)^2 + \frac12 (\d_r \psi)^2 + \frac1{p+1} c |\psi|^{p+1}
\biggr] \d_t - r^2 \bigl( \d_t \psi \d_r \psi \bigr)\d_r \] is continuously
differentiable and
\begin{align*}
 \Div \E & = r^2 \Bigl[ \Box \psi + c \psi |\psi|^{p-1} \Bigr] \d_t \psi
  + \frac1{p+1} r^2 (\d_t c) |\psi|^{p+1} \\
 & = \frac1{p+1} r^2 (\d_t c) |\psi|^{p+1}
\end{align*}
holds on $K^-$. Assume furthermore that $\d_t c \leq 0$ is non-positive
such that the same is true for $\Div \E$. This implies, recalling
Figure~\ref{fig:ct}, together with the assumption that the support $I$ of
the initial data is compactly contained in $\{1\} \times [0,\alpha_p[$, that
\begin{align*}
 0 & \geq \int_S (\Div \E) dt \wedge dr = \int_{\d S} i_\E (dt \wedge dr)
     = \int_{J \cup H} \Bigl[ (i_\E dt) dr - (i_\E dr) dt \Bigr] \\
   & = \int_J  r^2 \biggl[ \frac12 (\d_t \psi)^2 + \frac12 (\d_r \psi)^2 +
     \frac1{p+1} c |\psi|^{p+1} \biggr] dr \\
   & \quad - \frac{p-2}2 \int_I (i_\E dt) \circ \Phi(1+r,1-r) \cdot \biggl[
    \frac1{(1-r)^{p-1}} + \frac1{(1+r)^{p-1}} \biggr] dr \\
   & \quad - \frac{p-2}2 \int_I (i_\E dr) \circ \Phi(1+r,1-r) \cdot \biggl[
    \frac1{(1-r)^{p-1}} - \frac1{(1+r)^{p-1}} \biggr] dr,
\end{align*}
where $i_\E$ denotes the interior multiplication with the vector field $\E$.
The last two integrals over the interval $I$ depend only on the initial data
$\phi_0$ and $\phi_1$ on $I$ and are certainly finite. Thus
\begin{equation} \label{eq:Ebd}
 E_0:= \int_J  r^2 \biggl[ \frac12 (\d_t \psi)^2 + \frac12 (\d_r \psi)^2 +
 \frac1{p+1} c |\psi|^{p+1} \biggr] dr \leq C
\end{equation}
for a constant $C$.

The pseudo-energy $E_0$ now controls light-cone integrals of the quantity
$c|\psi|^{p+1}$ according to the following Proposition~\ref{lci}, a fact
that will be essential for proving boundedness of $\psi$ on the region
$K:=K^- \cap \{-1\leq t<0\}$.
\begin{Prop} \label{lci}
 Let $\psi \in C^2(K)$ be a solution of \eqref{eq:vwav} with initial data
 supported on $J$ satisfying the estimate~\eqref{eq:Ebd}. If $\d_t c \leq
 0$ on $K$ then \[ \frac1{p+1} \int_{-1}^{t_0} \bigl[ r^2 c |\psi|^{p+1}
 \bigr](s,t_0-s) ds \leq E_0\] holds for any $-1\leq t_0 < 0$.
\end{Prop}

\begin{proof}
It is useful to consider the functions \[ E_{t_0}(t) := \int_0^{t_0-t} (i_\E
dt)(t,r) dr \] for any $-1\leq t \leq t_0<0$. Then, for a fixed $-1<t_0<0$,
$E_{t_0}$ is continuously differentiable and
\begin{align*}
 E'_{t_0}(t) & = - (i_\E dt)(t,t_0-t) - \int_0^{t_0-t} \bigl[ \d_r(i_\E dr)
  \bigr] (t,r) dr + \int_0^{t_0-t} (\Div \E)(t,r) dr \\
 & = -(t_0-t)^2 \biggl[ \frac12 (\d_t \psi-\d_r \psi)^2 +  \frac1{p+1}
  c |\psi|^{p+1} \biggr] (t,t_0-t) + \int_0^{t_0-t} (\Div \E)(t,r) dr
\end{align*}
holds for all $-1 \leq t \leq t_0$. Hence, it follows that
\begin{align*}
 \Flux(-1,t_0) & := \int_{-1}^{t_0} i_\E (dt + dr) (s,t_0-s) ds \\
 & \phantom{:} = \int_{-1}^{t_0} (t_0-s)^2 \biggl[ \frac12 (\d_t \psi-\d_r
  \psi)^2 + \frac1{p+1} c |\psi|^{p+1} \biggr] (s,t_0-s) ds \\
 & \phantom{:} = - \int_{-1}^{t_0} E'_{t_0}(s)ds + \int_{-1}^{t_0}
  \int_0^{t_0-t} (\Div \E)(s,r) dr ds \\
 & \phantom{:} = E_{t_0}(-1) + \int_{-1}^{t_0} \int_0^{t_0-t} (\Div \E)(s,r)
  dr ds \leq E_0,
\end{align*}
since $E_{t_0}(t_0)=0$, $E_{t_0}(-1) \leq E_0$ and $\Div \E \leq 0$. But
this implies \[ \frac1{p+1} \int_{-1}^{t_0} (t_0-s)^2 \bigl[ c |\psi|^{p+1}
\bigr](s,t_0-s) ds \leq \Flux(-1,t_0) \leq E_0, \] and hence the claim.
\end{proof}

\section{Boundedness}

Using the light-cone estimate established in Proposition~\ref{lci} the
boundedness of $\psi$ in the region $K=K^- \cap \{-1 \leq t < 0\}$ can be
shown by a simple argument going back to Pecher \cite{Pecher}.

\begin{Theorem} \label{bound}
 Let $\psi \in C^2(K)$ be a solution of \eqref{eq:vwav} with $2 < p < 5$ and
 initial data supported on $J$ satisfying the estimate~\eqref{eq:Ebd}. If
 $c \geq 0$ is uniformly bounded and $\d_t c \leq 0$ on $K$ then $\psi$
 is uniformly bounded.
\end{Theorem}

\begin{proof}
Let $\psi_0 \in C^2(K)$ be a classical solution of the homogeneous
equation $\Box \psi_0 = 0$ with the same initial data as $\psi$. Then, for
a fixed $(t,r) \in K$, it holds that \[ |\psi-\psi_0|(t,r) \leq \frac1{2r}
\int_{-1}^t \int_{|t-s-r|}^{t-s+r} y c(s,y) |\psi|^p(s,y) dy ds. \] Due to the
fact that $2 < p < 5$ there exists a $q$ with $3/2 < q < (p+1)/(p-1)$.
Changing variables and applying H\"older's inequality yields
\begin{equation} \label{eq:hoelder}
\begin{split}
 \lefteqn{\frac1{2r} \int_{-1}^t \int_{|t-s-r|}^{t-s+r} y c(s,y) |\psi|^p(s,y)
  dy ds = \frac1{2r} \int_a^{t+r} \int_{-1}^b (t_0-s) c(s,t_0-s) |\psi|^p(s,t_0-s)
  ds dt_0 } \quad \\
 & \leq \frac1{2r} \biggl[ \int_a^{t+r} \int_{-1}^b (t_0-s)^2 c^q(s,t_0-s)
  |\psi|^{pq}(s,t_0-s) ds dt_0 \biggr]^\frac1q \biggl[ \int_a^{t+r}
  \int_{-1}^b (t_0-s)^\frac{q-2}{q-1} ds dt_0 \biggr]^\frac{q-1}q,
\end{split}
\end{equation}
where the abbreviations $a:=\max\{t-r,\,r-t-2\}$ and $b:=\bigl[
t_0+(t-r)\bigr]/2$ were introduced. Consider the first integral. Since $q>1$
and $c$ is bounded, so is $c^{q-1}$. Furthermore, $0<pq-(p+1)<q$, so that
\begin{align*}
 \lefteqn{\int_a^{t+r} \int_{-1}^b (t_0-s)^2 c^q(s,t_0-s) |\psi|^{pq}(s,t_0-s)
  ds dt_0 } \quad \\
 & \leq C \biggl( \sup_{-1\leq \tau\leq t} \|\psi(\tau,\cdot)\|_{L^\infty}
  \biggr)^{\gamma q} \int_a^{t+r} \int_{-1}^b (t_0-s)^2 c(s,t_0-s)
  |\psi|^{p+1}(s,t_0-s) ds dt_0,
\end{align*}
where $0<\gamma<1$ is such that $\gamma q = pq-(p+1)$. The integration
variable $t_0$ takes values $t_0 \geq a$ which ensures $-1 \leq b
\leq t_0$.  Thus, the integral in $s$ can be estimated by virtue of
Proposition~\ref{lci} to yield \[ \int_a^{t+r} \int_{-1}^b (t_0-s)^2
c(s,t_0-s) |\psi|^{p+1}(s,t_0-s) ds dt_0 \leq \int_{t-r}^{t+r} (p+1) E_0 dt_0
\leq 2rCE_0. \] The second integral in \eqref{eq:hoelder} can be calculated
directly to give \[ \int_a^{t+r} \int_{-1}^b (t_0-s)^\frac{q-2}{q-1} ds dt_0
\leq \frac{q-1}{2q-3} \int_{t-r}^{t+r} (1+t_0)^\frac{2q-3}{q-1} dt_0 \leq 2
r C \] because $q>3/2$. To sum up, the estimate \[ |\psi-\psi_0|(t,r) \leq
\frac C{2r} \biggl( \sup_{-1\leq \tau\leq t} \|\psi(\tau,\cdot)\|_{L^\infty}
\biggr)^\gamma \bigl( 2rE_0 \bigr)^\frac1q \bigl( 2r \bigr)^\frac{q-1}q = C
E_0^\frac1q \biggl( \sup_{-1\leq \tau\leq t} \|\psi(\tau,\cdot)\|_{L^\infty}
\biggr)^\gamma \] holds true for any $(t,r) \in K$. But since the
solution $\psi_0$ of the homogeneous equation is clearly bounded and
$0 < \gamma < 1$ this estimate implies the boundedness of $\psi$ itself
uniformly on $K$, for the function $t \mapsto \sup_{-1\leq \tau\leq t}
\|\psi(\tau,\cdot)\|_{L^\infty}$ is continuous and finite at $t=-1$.
\end{proof}

With the function $\psi$ bounded on $K$ the decay of $\phi$ in the future of
$I$ follows immediately.

\begin{Cor} \label{decay}
 Let $\phi$ be a classical solution of the wave equation~\eqref{eq:wav}
 for $3\leq p < 5$ with initial data $\phi_0 \in C^2(\Rp)$ and $\phi_1
 \in C^1(\Rp)$ given at $t=1$ and which exists globally towards the
 future. Assume that the support of $\phi_0$ and $\phi_1$ is contained in
 the compact interval $I \subset [0,\,\alpha_p[$. Then there is a constant $C
 \in \Rp$ such that $\phi$ satisfies the decay estimate 
 \begin{equation}\label{eq:main-estimate-Cor}
   |\phi(t,r)| \leq \frac C{(1+t+r)(1+t-r)^{p-2}}
 \end{equation}
 for all $t\geq 1$ and $0\leq r \leq t$.
\end{Cor}

\begin{proof}
Given such a solution $\phi$ of class $C^2$, it was shown in
section~\ref{sec:ct} that its conformal transformation $\psi =
\Phi_* (\Omega^{-1} \phi)$ is a classical solution of the wave
equation~\eqref{eq:vwav} on the future of the curve $H$ in $K^-$,
cf. Figure~\ref{fig:ct}, with \[ c = \frac1{(p-2)^2}\Phi_* (\Omega u v)^{p-1}
\] according to equation~\eqref{eq:ctwav}. Now, for $v<u<0$ \[ \Phi_* (\Omega
u v) = \frac{u-v}{(-v)^\frac1{p-2} - (-u)^\frac1{p-2}}, \] so that $c$
is bounded on the future of $H$ in $K^-$ owing to $p\geq 3$. Furthermore,
\[ \d_t \bigl[ \Phi_* (\Omega u v) \bigr] = \frac1{p-2} \frac{\Phi_*
(\Omega u v)^2}{2r} \Bigl[ (-v)^{-\frac{p-3}{p-2}} - (-u)^{-\frac{p-3}{p-2}}
\Bigr], \] which implies $\d_t c \leq 0$ on $K^-$ again by reason of $p\geq
3$. Moreover, as detailed in section~\ref{sec:evolE}, the assumptions on the
support of $\phi_0$ and $\phi_1$ guarantee that $J$ is compactly contained
in $\{-1\} \times [0,1[$ and that the estimate~\eqref{eq:Ebd} holds. Thus,
Theorem~\ref{bound} applies and yields boundedness of $\psi$ on $K$. Since
$\psi$ is also bounded on the compact region $S$ it follows that $\psi$
is bounded on the whole future of $H$ in $K^-$, say $|\psi| \leq C$. But
then \[ |\phi(t,r)| = \Omega(t,r) |\Phi^* \psi(t,r)| \leq C \Omega(t,r) \]
on the whole future of $I$ in $K^+$, while  
\[ \Omega(t,r) = \frac1{2rv^{p-2}} \biggl[
1- \Bigl(\frac vu \Bigr)^{p-2} \biggr] \leq \frac{p-2}{2rv^{p-2}} \biggl(
1-\frac vu \biggr) = \frac{p-2}{uv^{p-2}}. \] 
Since in this region $u\geq 1$ and $v\geq 1-\alpha_p$ and outside, for
$v<1-\alpha_p$, the solution vanishes identically, the bound for $\phi(t,r)$
can be written in the regularized form \eqref{eq:main-estimate-Cor}.

\end{proof}

\section{Discussion}

As already noted in the introduction, the conformal transformations which we consider for the radial problem are not conformal for the full problem, beyond the spherical symmetry, unless $p=3$. However, for the full problem with $p>3$ we can still apply the conformal ($p=3$) transformation and obtain a global pointwise decay $1/(uv)$. The decay rate is then not optimal but can act as a prerequisite for a more refined asymptotic analysis. Indeed, that decay rate is sufficient for the solution to become small at large times in the sense which allows a perturbative analysis with methods similar to those developed in \cite{NS-Tails}. We want to address this issue in a forthcoming publication.

It seems, at least as far as the spherical symmetry is concerned, that our method can be applied in higher than $n=3$ odd dimensions.

Also the linear wave equations with strong positive potentials 
$$ \d_t^2 \phi - \Delta \phi + V(x) \phi = 0 $$
having prescribed decay at spatial infinity $V(x)\sim 1/|x|^k$ still lack a sharp pointwise decay estimate and seem to be treatable with our method after minor modifications. Such equations have a positive definite energy, too, and can be conformally transformed to a form analogous to \eqref{eq:vwav} with a regular positive function $c$. This idea shall be addressed in another publication.

\begin{acknowledgments}
One of the authors (NS) wants to express his gratitude to the Mittag-Leffler Institute and to the organizers of the workshop ``Geometry, Analysis, and General Relativity'' (fall 2008) for hospitality and fantastic, creative and relaxed atmosphere during the work which resulted in this article.
\end{acknowledgments}

\bibliography{biblio}
\bibliographystyle{alpha}

\end{document}